\documentclass[a4paper]{article}
\usepackage{fullpage}
\usepackage{amsmath}
\usepackage{amssymb}
\usepackage{amsthm}
\usepackage{graphicx}
\usepackage{bbm}
\usepackage{enumerate}
\usepackage{microtype}
\usepackage{tocbibind}
\usepackage{xcolor}
\usepackage{indentfirst}
\usepackage[none]{hyphenat}
\usepackage[colorlinks=true,urlcolor=blue,linkcolor=blue,plainpages=false,pdfpagelabels]{hyperref}
\allowdisplaybreaks

\newtheorem{theorem}{Theorem}[section]
\newtheorem{proposition}[theorem]{Proposition}

\newtheorem{corollary}[theorem]{Corollary}
\newtheorem{remark}[theorem]{Remark}

\newtheorem{definition}[theorem]{Definition}

\newcommand{\N}{\mathbb{N}}

\newcommand{\R}{\mathbb{R}}

\newcommand{\eps}{\varepsilon}

\newcommand{\se}{\subseteq}

\newcommand{\wt}[1]{\widetilde{#1}}

\DeclareMathOperator{\argmax}{argmax}
\DeclareMathOperator{\argmin}{argmin}
\DeclareMathOperator{\Prox}{Prox}
\title{On the uniform convexity of the squared distance}
\author{Andrei Sipo\c s${}^{a,b,c}$\\[2mm]
\footnotesize ${}^a$Research Center for Logic, Optimization and Security (LOS), Department of Computer Science,\\
\footnotesize Faculty of Mathematics and Computer Science, University of Bucharest,\\
\footnotesize Academiei 14, 010014 Bucharest, Romania\\[1mm]
\footnotesize ${}^b$Simion Stoilow Institute of Mathematics of the Romanian Academy,\\
\footnotesize Calea Grivi\c tei 21, 010702 Bucharest, Romania\\[1mm]
\footnotesize ${}^c$Institute for Logic and Data Science,\\
\footnotesize Popa Tatu 18, 010805 Bucharest, Romania\\[2mm]
\footnotesize Email: andrei.sipos@fmi.unibuc.ro\\
}
\date{}

\begin{document}

\maketitle

\begin{abstract}
In 1983, Z\u alinescu showed that the squared norm of a uniformly convex normed space is uniformly convex on bounded subsets. We extend this result to the metric setting of uniformly convex hyperbolic spaces. We derive applications to the convergence of shadow sequences and to proximal minimization.


\end{abstract}

\section{Introduction}

In 1936, Clarkson \cite{Cla36} introduced the notion of uniform convexity for normed spaces. Over the years, this has proven to be a highly fruitful abstraction, along with its companion, the `dual' notion of uniform smoothness, both of them jointly capturing the most relevant properties of the most widely known examples of non-Hilbert Banach spaces: the $L^p$ spaces.

In 1983, Z\u alinescu \cite{Zal83} showed that uniformly convex normed spaces are also uniformly convex in a different sense, namely their squared norm is a uniformly convex function on bounded subsets. Again, this additional property has repeatedly shown its worth in the study of nonlinear operators, as seen in the papers \cite{Xu,XuRoach} and the ones citing those.

This idea of studying nonlinear operators, in contrast to the usual focus of functional analysis on (continuous) linear ones, has naturally given way in recent decades to its generalization to `nonlinear spaces', metric spaces which retain some of the additional structure given by the linearity of a normed space. Such a kind of space, aiming to preserve the idea of convexity, was introduced by  Kohlenbach in 2005 \cite{Koh05} under the name of {\it $W$-hyperbolic spaces} (see the next section for a definition and \cite[pp. 384--388]{Koh08} for a detailed discussion on the relationship between various definitions of hyperbolicity).

Kohlenbach was also motivated by the research program of `proof mining', which he largely led during that time, a program which aimed to apply tools from proof theory (a branch of mathematical logic) to proofs in mainstream mathematics in order to uncover information which may not be readily apparent. This logical study of ordinary mathematical proofs usually leads the practitioner towards proof-theoretically natural and flexible axioms for various classes of mathematical structures, this also being the case with this notion of $W$-hyperbolicity. Proof mining, as pointed out below, will also drive some of our own choices of presenting our results. For more information about proof mining, see Kohlenbach's 2008 monograph \cite{Koh08} or his more recent ICM survey \cite{Koh19}.

The proper generalization of uniform convexity to this $W$-hyperbolic setting was found by Leu\c stean \cite{Leu07,Leu10} in the form of $UCW$-hyperbolic spaces. Some of the properties of uniformly convex normed spaces were found to generalize naturally to $UCW$-hyperbolic spaces in those papers, as well as in \cite{KohLeu10}. Typical examples of non-normed $UCW$-hyperbolic spaces are provided by the class of CAT(0) spaces, the canonical nonlinear generalization of inner product spaces, as studied in metric geometry. For a recent example of a $UCW$-hyperbolic space which is neither a CAT(0) space, nor a convex subset of a normed space, see \cite[Section 4]{PinSip}.

All these classes of nonlinear spaces, with their basic definitions and properties, may be found in Section~\ref{sec:prelim}. In addition, we show that the fixed point sets of (asymptotically) nonexpansive mappings on $UCW$-hyperbolic spaces are convex in a more general, quantitative way, as inspired by proof mining. (Proof mining can frequently help in finding the most natural logical form for a given property.)

The main result of this paper, as shown in Section~\ref{sec:main}, generalizes Z\u alinescu's result to $UCW$-hyperbolic spaces, with the squared norm being superseded by the (section of the) squared distance. The result is given, again, in a quantitative way, with a formula for the modulus for the uniform convexity of the squared distance, as inspired by the proof-theoretic analysis of Z\u alinescu's result in \cite{BacKoh18}.

The final two sections are devoted to applications, thus vindicating the study of this property in this nonlinear setting. Section~\ref{sec:shadow} extends to complete $UCW$-hyperbolic spaces the classical results for the convergence of shadow sequences associated to (quasi)-Fej\'er monotone sequences. These convergence results are given in terms of `rates of metastability' (a concept originating in proof mining, to be defined and motivated in said section). Section~\ref{sec:proximal} shows for the first time that one can properly define proximal mappings, and is thus a pilot study in convex optimization in complete $UCW$-hyperbolic spaces.

\section{Preliminaries}\label{sec:prelim}

\subsection{Nonlinear spaces}

One says that a metric space $(X,d)$ is {\it geodesic} if for any two points $x$, $y \in X$ there is a {\it geodesic} that joins them, i.e.\ a mapping $\gamma : [0,1] \to X$ such that $\gamma(0)=x$, $\gamma(1)=y$ and for any $t$, $t' \in [0,1]$ we have that
$$d(\gamma(t),\gamma(t')) = |t-t'| d(x,y).$$

Generally, geodesics need not be unique, and, thus, one is led to the following definition, which considers geodesics on a metric space as an additional structure\footnote{The convexity function $W$ was first considered by Takahashi in \cite{Tak70} where a triple $(X,d,W)$ satisfying $(W1)$ is called a convex metric space. The notion used here, frequently considered nowadays to be the nonlinear generalization of convexity in normed spaces, was introduced by Kohlenbach in \cite{Koh05}. As said in the Introduction, see \cite[pp. 384--388]{Koh08} for a detailed discussion on the relationship between various definitions of hyperbolicity and on the proof-theoretical considerations that ultimately led to the adoption of this one. We only mention here that this notion is more general than that of hyperbolic spaces in the sense of Reich and Shafrir \cite{ReiSha90}, and slightly more restrictive than the setting due to Goebel and Kirk \cite{GoeKir83} of spaces of hyperbolic type. } (and not as a property), which is required to satisfy additional properties.

\begin{definition}
A {\bf $W$-hyperbolic space} is a triple $(X,d,W)$ where $(X,d)$ is a metric space and $W: X^2 \times [0,1] \to X$ such that, for all $x$, $y$, $z$, $w \in X$ and $\lambda$, $\mu \in [0,1]$, we have that
\begin{enumerate}[(W1)]
\item $d(z,W(x,y,\lambda)) \leq (1-\lambda)d(z,x) + \lambda d(z,y)$;
\item $d(W(x,y,\lambda),W(x,y,\mu)) = |\lambda-\mu|d(x,y)$;
\item $W(x,y,\lambda)=W(y,x,1-\lambda)$;
\item $d(W(x,z,\lambda),W(y,w,\lambda))\leq(1-\lambda) d(x,y) + \lambda d(z,w)$.
\end{enumerate}
\end{definition}

Clearly, any normed space may be made into a $W$-hyperbolic space in a canonical way. 

A subset $C$ of a $W$-hyperbolic space $(X,d,W)$ is called {\it convex} if, for any $x$, $y \in C$ and $\lambda \in [0,1]$, $W(x,y,\lambda) \in C$. If $(X,d,W$) is a $W$-hyperbolic space, $x$, $y \in X$ and $\lambda \in [0,1]$, we denote the point $W(x,y,\lambda)$ by $(1-\lambda)x+\lambda y$. We will, also, mainly write $\frac{x+y}2$ for $\frac12 x + \frac12 y$.

\begin{proposition}\label{quad}
Let $(X,d,W)$ be a $W$-hyperbolic space, $\lambda \in [0,1]$ and $x$, $y$, $a \in X$. Then
$$d^2\left((1-\lambda)x+\lambda y,a\right) \leq (1-\lambda)d^2(x,a)+\lambda d^2(y,a).$$
\end{proposition}

\begin{proof}
Using $(W1)$, we get that
$$d^2\left((1-\lambda)x+\lambda y,a\right) \leq \left((1-\lambda)d(x,a)+\lambda d(y,a)\right)^2 \leq (1-\lambda)d^2(x,a)+\lambda d^2(y,a),$$
the last inequality following from the convexity of the square function.
\end{proof}

As per \cite{Koh05,Leu07}, a particular nonlinear class of $W$-hyperbolic spaces is the one of CAT(0) spaces, introduced by A. Aleksandrov \cite{Ale51} and named as such by M. Gromov \cite{Gro87}, which may be defined, by the discussion in \cite[pp. 387--388]{Koh08}, as those $W$-hyperbolic spaces $(X,d,W)$ such that, for any $a$, $x$, $y \in X$,
\begin{equation}\label{cat0}
d^2\left(\frac{x+y}2,a\right) \leq \frac12 d^2(x,a) + \frac12 d^2(y,a) - \frac14 d^2(x,y).
\end{equation}

In particular, any inner product space is a CAT(0) space, and, indeed, as experience shows, CAT(0) spaces may be regarded as the rightful nonlinear generalization of inner product spaces.

Uniform convexity was first introduced by Leu\c stean \cite{Leu07,Leu10}  in this hyperbolic setting, as motivated by \cite[p. 105]{GoeRei84}. The monotonicity condition was justified by proof-theoretical considerations.

\begin{definition}
If $(X,d,W)$ is a $W$-hyperbolic space, then a {\bf modulus of uniform convexity} for $(X,d,W)$ is a function $\eta :(0, \infty) \times (0,2] \to (0,1]$ such that, for any $r >0$, any $\eps \in (0,2]$ and any $a$, $x$, $y \in X$ with $d(x,a) \leq r$, $d(y,a) \leq r$, $d(x,y) \geq \eps r$, we have that
$$d\left(\frac{x+y}2,a\right) \leq (1-\eta(r,\eps))r.$$
We call the modulus {\bf monotone} if, for any $r$, $s >0$ with $s \leq r$ and any $\eps \in (0,2]$, we have that $\eta(r,\eps) \leq \eta(s,\eps)$.

A {\bf $UCW$-hyperbolic space} is a $W$-hyperbolic space that admits a monotone modulus of uniform convexity.
\end{definition}

As remarked in \cite[Proposition 2.6]{Leu07}, CAT(0) spaces are $UCW$-hyperbolic spaces having as a modulus of uniform convexity the simple function $(r,\eps) \mapsto \eps^2/8$.

In a metric space $(X,d)$, a nonempty subset $S$ of $X$ is a {\it Chebyshev set} if, for any $x \in X$, there exists a unique $y \in S$ such that, for any $z \in S$, $d(x,y) \leq d(x,z)$ -- this defines the {\it projection mapping} $P_S:X \to S$. By \cite[Proposition 2.4]{Leu10}, any closed, convex, nonempty subset of a complete $UCW$-hyperbolic space is a Chebyshev set. Also, in any metric space $(X,d)$, we say that a sequence $(x_n)$ is {\it Fej\'er monotone} with respect to a subset $S$ of $X$ if, for all $n \in \N$ and $q \in S$, $d(x_{n+1},q)\leq d(x_n,q)$.

For a self-mapping $T$ of a metric space $(X,d)$, we denote by $Fix(T)$ its fixed point set.

\subsection{Nonexpansive mappings}

A self-mapping $T$ of a metric space $(X,d)$ is called {\it nonexpansive} if, for any $x$, $y \in X$, $d(Tx,Ty) \leq d(x,y)$ -- it is immediate that, for such a mapping, $Fix(T)$ is closed. By \cite[Corollary 3.7]{Leu10}, if $X$ is a bounded complete nonempty $UCW$-hyperbolic space, then any nonexpansive self-mapping of $X$ has a fixed point.

The following proposition, a quantitative generalization of the convexity of the fixed point set as inspired by proof mining, uses a quantitative version of the argument used in the proof of the strict convexity of $UCW$-hyperbolic spaces \cite[Proposition 5]{Leu07}. The proposition is a significant nonlinear generalization of the argument in \cite[Lemma 2.3]{Koh11}.

\begin{proposition}
Let $\eta :(0,\infty) \times (0,2] \to (0,\infty)$ be monotone and let $(X,d,W)$ be a $UCW$-hyperbolic space admitting $\eta$ as a modulus. Let $T: X \to X$ be nonexpansive, $x$, $y \in X$, $\eps > 0$, $b > 0$ with $d(x,y) \leq b$. Let $\lambda \in [0,1]$ and set $z:=(1-\lambda)x+\lambda y$. Put
$$\delta:=\min\left( \frac\eps2,b,\frac{\eps^2}{16b}\eta\left(2b,\min\left(\frac{\eps}{2b},2\right)\right)\right).$$
Assume that $d(x,Tx) \leq \delta$ and that $d(y,Ty) \leq \delta$. Then $d(z,Tz) \leq \eps$.
\end{proposition}

\begin{proof}
First of all, note that
$$d(Tz,x) \leq d(Tz,Tx) + \delta \leq d(z,x) + \delta = \lambda d(x,y) + \delta$$
and, similarly, that
$$d(Tz,y) \leq (1-\lambda)d(x,y) + \delta,$$
so
$$d(z,Tz) \leq d(z,x) + d(Tz,x) \leq 2\lambda d(x,y) + \delta \leq 3b,$$
i.e. we are only interested in cases where $\frac\eps{2b} \leq 2$, where
$$\min\left(\frac{\eps}{2b},2\right) = \frac{\eps}{2b}.$$

{\bf Case I.} $\lambda \leq \frac{\eps}{4b}$.

In this case, as before,
$$d(z,Tz) \leq 2\lambda d(x,y) + \delta \leq \eps.$$

{\bf Case II.} $1- \lambda \leq \frac{\eps}{4b}$. Similarly, using that
$$d(z,Tz) \leq d(z,y) + d(Tz,y) \leq 2(1-\lambda)d(x,y) + \delta \leq \eps.$$

{\bf Case III.} $d(x,y) \leq \frac\eps4$.

In this case,
$$d(z,Tz) \leq 2\lambda d(x,y) + \delta \leq \eps.$$

{\bf Case IV.} None of the above cases hold, so $\lambda > \frac{\eps}{4b}$, $1- \lambda > \frac{\eps}{4b}$ and $d(x,y) > \frac\eps4$.

Put $r_1:=\lambda d(x,y)+\delta$ and $r_2:=(1-\lambda )d(x,y) + \delta$.

We have that $r_1 \leq 2b$, so
$$\eta\left(r_1,\frac{\eps}{2b}\right) \geq \eta\left(2b,\frac{\eps}{2b}\right),$$
and $r_1 \geq \lambda d(x,y) > \frac{\eps^2}{16b}$.

Assume towards a contradiction that $d(z,Tz) > \eps \geq \frac{\eps}{2b} \cdot r_1$. Since $d(z,x) \leq r_1$ and $d(Tz,x) \leq r_1$, we have that
$$d\left(\frac{z+Tz}2,x\right) \leq \left(1-\eta\left(r_1,\frac{\eps}{2b}\right)\right) \cdot r_1 \leq \left(1-\eta\left(2b,\frac{\eps}{2b}\right)\right) \cdot r_1 < r_1 - \frac{\eps^2}{16b}\eta\left(2b,\frac{\eps}{2b}\right) \leq r_1 - \delta.$$

Similarly, we obtain that
$$d\left(\frac{z+Tz}2,y\right) < r_2 -\delta,$$
so
$$d(x,y) \leq d\left(\frac{z+Tz}2,x\right) + d\left(\frac{z+Tz}2,y\right) < r_1 + r_2 - 2\delta = d(x,y),$$
a contradiction.
\end{proof}

\begin{corollary}
Let $(X,d,W)$ be a $UCW$-hyperbolic space and $T: X \to X$ be a nonexpansive mapping. Then $Fix(T)$ is convex.
\end{corollary}

\subsection{Asymptotically nonexpansive mappings}

If $(\delta_n)$ is a sequence of non-negative reals such that $\delta_n \to 0$ (sometimes, the stronger condition $\sum \delta_n < \infty$ will be imposed), then a self-mapping $T$ of a metric space $(X,d)$ is called {\it asymptotically nonexpansive} w.r.t. $(\delta_n)$ if, for any $x$, $y \in X$ and $n \in \N$, $d(T^nx,T^ny) \leq (1+\delta_n)d(x,y)$ -- it is immediate that, for such a mapping, $Fix(T)$ is closed. By \cite[Theorem 3.3]{KohLeu10}, if $X$ is a bounded complete nonempty $UCW$-hyperbolic space, then any asymptotically nonexpansive self-mapping of $X$ has a fixed point.

\begin{proposition}
Define, for any suitable $b$, $\eta$, $u$, $B$, $\eps$:
\begin{align*}
\Theta_{b,\eta}(\eps)&:=\frac12 \cdot \min\left( \frac\eps2,b,\frac{\eps^2}{16b}\eta\left(2b,\min\left(\frac{\eps}{2b},2\right)\right)\right)\\
\gamma_{b,\eta,u}(\eps)&:=u\left(\frac{\Theta_{b,\eta}(\eps)}b\right)\\
N_{b,\eta,u,B}(\eps)&:=\gamma_{b,\eta,u}\left(\frac\eps{2+B}\right)\\
\Omega_{b,\eta,u,B}(\eps)&:=\frac1{(N_{b,\eta,u,B}(\eps)+1)(1+B)} \cdot \Theta_{b,\eta}\left(\frac\eps{2+B}\right).
\end{align*}
Let $\eta :(0,\infty) \times (0,2] \to (0,\infty)$ be monotone and let $(X,d,W)$ be a $UCW$-hyperbolic space admitting $\eta$ as a modulus. Let $(\delta_n)$ be a sequence of non-negative reals and $u : (0,\infty) \to \N$ be such that, for every $n \geq u(\eps)$, $\delta_n \leq \eps$. Let $T:X \to X$ be asymptotically nonexpansive w.r.t. $(\delta_n)$. Let $x$, $y \in X$, $b > 0$ with $d(x,y) \leq b$. Let $\lambda \in [0,1]$ and set $z:=(1-\lambda)x+\lambda y$.
\begin{enumerate}[(i)]
\item Let $\eps > 0$ and $n \geq \gamma_{b,\eta,u}(\eps)$. Assume that $d(x,T^nx) \leq \Theta_{b,\eta}(\eps)$ and that $d(y,T^ny) \leq \Theta_{b,\eta}(\eps)$. Then $d(z,T^nz) \leq \eps$.
\item Let $\eps > 0$ and $B \geq 0$ be such that, for every $n$, $\delta_n \leq B$. Assume that $d(x,Tx) \leq \Omega_{b,\eta,u,B}(\eps)$ and that $d(y,Ty) \leq \Omega_{b,\eta,u,B}(\eps)$. Then $d(z,Tz) \leq \eps$.
\end{enumerate}
\end{proposition}

\begin{proof}
\begin{enumerate}[(i)]
\item First of all, put
$$\delta:= \min\left( \frac\eps2,b,\frac{\eps^2}{16b}\eta\left(2b,\min\left(\frac{\eps}{2b},2\right)\right)\right),$$
so $\Theta_{b,\eta}(\eps) = \delta/2$ and $\gamma_{b,\eta,u}(\eps)=u\left(\frac\delta{2b}\right)$.

Now, note that (using in the third inequality that $n \geq u(\delta/2b)$)
$$d(T^nz,x) \leq d(T^nz,T^nx) + \delta/2 \leq (1+\delta_n)d(z,x) + \delta/2 \leq d(z,x) + \delta/2 + \delta/2 = \lambda d(x,y) + \delta$$
and, similarly,
$$d(T^nz,y) \leq (1-\lambda)d(x,y) + \delta,$$
so
$$d(z,T^nz) \leq d(z,x) + d(T^nz,x) \leq 2\lambda d(x,y) + \delta \leq 3b,$$
i.e. we are only interested in cases where $\frac\eps{2b} \leq 2$, where
$$\min\left(\frac{\eps}{2b},2\right) = \frac{\eps}{2b}.$$

{\bf Case I.} $\lambda \leq \frac{\eps}{4b}$.

In this case, as before,
$$d(z,T^nz) \leq 2\lambda d(x,y) + \delta \leq \eps.$$

{\bf Case II.} $1- \lambda \leq \frac{\eps}{4b}$. Similarly, using that
$$d(z,T^nz) \leq d(z,y) + d(T^nz,y) \leq 2(1-\lambda)d(x,y) + \delta \leq \eps.$$

{\bf Case III.} $d(x,y) \leq \frac\eps4$.

In this case,
$$d(z,T^nz) \leq 2\lambda d(x,y) + \delta \leq \eps.$$

{\bf Case IV.} None of the above cases hold, so $\lambda > \frac{\eps}{4b}$, $1- \lambda > \frac{\eps}{4b}$ and $d(x,y) > \frac\eps4$.

Put $r_1:=\lambda d(x,y)+\delta$ and $r_2:=(1-\lambda )d(x,y) + \delta$.

We have that $r_1 \leq 2b$, so
$$\eta\left(r_1,\frac{\eps}{2b}\right) \geq \eta\left(2b,\frac{\eps}{2b}\right),$$
and $r_1 \geq \lambda d(x,y) > \frac{\eps^2}{16b}$.

Assume towards a contradiction that $d(z,T^nz) > \eps \geq \frac{\eps}{2b} \cdot r_1$. Since $d(z,x) \leq r_1$ and $d(T^nz,x) \leq r_1$, we have that
$$d\left(\frac{z+T^nz}2,x\right) \leq \left(1-\eta\left(r_1,\frac{\eps}{2b}\right)\right) \cdot r_1 \leq \left(1-\eta\left(2b,\frac{\eps}{2b}\right)\right) \cdot r_1 < r_1 - \frac{\eps^2}{16b}\eta\left(2b,\frac{\eps}{2b}\right) \leq r_1 - \delta.$$

Similarly, we obtain that
$$d\left(\frac{z+T^nz}2,y\right) < r_2 -\delta,$$
so
$$d(x,y) \leq d\left(\frac{z+T^nz}2,x\right) + d\left(\frac{z+T^nz}2,y\right) < r_1 + r_2 - 2\delta = d(x,y),$$
a contradiction.
\item Set $n:=N_{b,\eta,u,B}(\eps)$. Now,
$$d(x,T^nx) \leq \sum_{i \in [0,n)} d(T^ix,T^{i+1}x) \leq \sum_{i \in [0,n)}(1+\delta_i)d(x,Tx) \leq \Omega_{b,\eta,u,B}(\eps) \cdot n \cdot (1+B) \leq \Theta_{b,\eta}\left(\frac\eps{2+B}\right).$$
Similarly, we get that
$$d(x,T^{n+1}x) \leq \Omega_{b,\eta,u,B}(\eps) \cdot (n+1) \cdot (1+B) = \Theta_{b,\eta}\left(\frac\eps{2+B}\right)$$
and also that $d(y,T^ny) \leq \Theta_{b,\eta}\left(\frac\eps{2+B}\right)$ and $d(y,T^{n+1}y) \leq \Theta_{b,\eta}\left(\frac\eps{2+B}\right)$.

Since both $n$ and $n+1$ are greater or equal to $\gamma_{b,\eta,u}\left(\frac\eps{2+B}\right)$, we get, using (i), that $d(z,T^nz) \leq \frac\eps{2+B}$ and $d(z,T^{n+1}z) \leq \frac\eps{2+B}$. We have that
$$d(Tz,T^{n+1}z) \leq (1+\delta_1)d(z,T^nz) \leq (1+B) \cdot \frac\eps{2+B}$$
and, thus, that
$$d(z,Tz) \leq d(z,T^{n+1}z) + d(Tz,T^{n+1}z) \leq \frac\eps{2+B} + (1+B) \cdot \frac\eps{2+B} = \eps,$$
so we are done.
\end{enumerate}
\end{proof}

\begin{corollary}
Let $(X,d,W)$ be a $UCW$-hyperbolic space and $T: X \to X$ be an asymptotically nonexpansive mapping. Then $Fix(T)$ is convex.
\end{corollary}

\subsection{On the uniform convexity of the squared distance}

The recent interest in proof mining towards the uniform convexity on bounded subsets of the squared norm of a uniformly convex normed space stemmed from its crucial use in the quantitative study by Kohlenbach and the author \cite{KohSip21} of the celebrated theorem of Reich \cite{Rei80} concerning strong convergence of approximating curves for resolvents of accretive operators. We stated it below in the metric manner which will be of concern to us, as it was for the first time explicitly formulated in \cite{PinSip}.

\begin{definition}[{\cite[Definition 2.6]{PinSip}}]
Let $\psi :(0, \infty) \times (0,\infty) \to (0,\infty)$. We say that a $W$-hyperbolic space $(X,d,W)$ {\bf has property $(G)$ with modulus $\psi$} if, for any $r$, $\eps >0$ and any $a$, $x$, $y \in X$ with $d(x,a) \leq r$, $d(y,a) \leq r$, $d(x,y) \geq \eps $, we have that
$$d^2\left(\frac{x+y}2,a\right) \leq \frac12 d^2(x,a) + \frac12 d^2(y,a) - \psi(r,\eps).$$
\end{definition}

For uniformly convex normed spaces, this property essentially goes back to Z\u alinescu \cite[Section 4]{Zal83} -- see also \cite{Xu,XuRoach} for more variations on this theme, and see \cite[Proposition 2.4]{KohSip21} (an instantiation of the more general \cite[Proposition 3.2]{BacKoh18}) for a formula for this modulus which is easily computable in terms of the modulus of uniform convexity.

\begin{proposition}
CAT(0) spaces have property $(G)$.
\end{proposition}

\begin{proof}
Let $X$ be a CAT(0) space and $r$, $\eps >0$, $a$, $x$, $y \in X$ with $d(x,a) \leq r$, $d(y,a) \leq r$, $d(x,y) \geq \eps $.
By \eqref{cat0}, we have that
$$d^2\left(\frac{x+y}2,a\right) \leq \frac12 d^2(x,a) + \frac12 d^2(y,a) - \frac14 d^2(x,y) \leq \frac12 d^2(x,a) + \frac12 d^2(y,a) -\frac{\eps^2}4,$$
and we are done.
\end{proof}

Recently, Pinto \cite{Pin24} extended the aforementioned analysis of Kohlenbach and the author of Reich's theorem to `uniformly smooth hyperbolic spaces' which are in addition $UCW$-hyperbolic. At the time of the writing of \cite{Pin24}, it was not known whether, generally, $UCW$-hyperbolic spaces satisfy property $(G)$; however, Pinto identified in \cite[Proposition 5.4]{Pin24} (see also \cite[Proposition 2.5]{Sip21}) a weaker property of them, called property $(M)$ in \cite[Definition 2.8]{PinSip}, as being enough for the proof to go through.

We now move on to showing that property $(G)$ does indeed hold in $UCW$-hyperbolic spaces, as well as justifying its study by giving applications which rely on it.

\section{Main result}\label{sec:main}

The following is the generalization to $UCW$-hyperbolic spaces of \cite[Proposition 2.4]{KohSip21}, which, as said earlier, is an instantiation of the more general \cite[Proposition 3.2]{BacKoh18}.

\begin{theorem}\label{mainthm}
Let $\eta :(0,\infty) \times (0,2] \to (0,\infty)$ be monotone and let $(X,d,W)$ be a $UCW$-hyperbolic space admitting $\eta$ as a modulus. For any suitable $r$, $\eps$, put
$$\psi_{\eta}(r,\eps):= \min \left( \frac{\left(\min\left(\frac\eps2, \frac{\eps^2}{96r}\eta^2\left(r,\min\left(\frac{\eps}{2r},2\right)\right)\right)\right)^2}4, \frac{\eps^2}{32}\eta^2\left(r,\min\left(\frac{\eps}{2r},2\right)\right) \right).$$
Then $(X,d,W)$ has property $(G)$ with modulus $\psi_\eta$.
\end{theorem}

\begin{proof}
Let $r$, $\eps>0$ and $a$, $x$, $y \in X$ with $d(x,a)\leq r$, $d(y,a) \leq r$ and $d(x,y) \geq \eps$. We want to show that
$$d^2\left(\frac{x+y}2,a\right) \leq \frac12 d^2(x,a) + \frac12 d^2(y,a) - \psi_\eta(r,\eps).$$
First of all, note that
$$\eps \leq d(x,y) \leq d(x,a) + d(y,a) \leq 2r,$$
i.e. we are only interested in cases where $\frac\eps{2r} \leq 2$, where
$$\min\left(\frac{\eps}{2r},2\right) = \frac{\eps}{2r}.$$

Put
$$\gamma := \frac{\eps^2}{32}\eta^2\left(r,\frac{\eps}{2r}\right), \quad \delta:=\min\left(\frac\eps2,\frac{\gamma}{3r}\right) = \min\left(\frac\eps2, \frac{\eps^2}{96r}\eta^2\left(r,\frac\eps{2r}\right)\right),$$
so that $\delta \leq \eps/2 \leq r$ and $\psi_\eta(r,\eps)=\min\left(\frac{\delta^2}4,\gamma\right)$. Also, put $\alpha:=d(x,a)$ and $\beta:=d(y,a)$ and assume w.l.o.g. that $\alpha \leq \beta$. Assume towards a contradiction that
$$d^2\left(\frac{x+y}2,a\right) > \frac12 d^2(x,a) + \frac12 d^2(y,a) - \psi_\eta(r,\eps).$$
Using $(W1)$, we have that
$$  \left(\frac{d(x,a)+d(y,a)}2\right)^2\geq d^2\left(\frac{x+y}2,a\right)> \frac12 d^2(x,a) + \frac12 d^2(y,a) - \frac{\delta^2}4,$$
from which we get that $\beta -\alpha < \delta$.

Put $z:=\left(1-\frac\alpha\beta\right)a + \frac\alpha\beta y$, so $d(z,a)=\alpha$ and $d(z,y)=\beta-\alpha < \delta$. Using $(W4)$, we have that
$$d\left(\frac{x+z}2,\frac{x+y}2\right) \leq \frac12d(z,y) \leq d(z,y)  <\delta,$$
so
$$d\left(\frac{x+y}2,a\right) \leq d\left(\frac{x+z}2,a\right) + d\left(\frac{x+z}2,\frac{x+y}2\right) < d\left(\frac{x+z}2,a\right) +\delta.$$
Now,
\begin{align*}
d^2(x,a) - \gamma &\leq d^2(x,a) - \psi_\eta(r, \eps)\\
 &\leq \frac12 d^2(x,a) + \frac12 d^2(y,a) - \psi_\eta(r,\eps) \\
&< d^2\left(\frac{x+y}2,a\right) \\
&< \left( d\left(\frac{x+z}2,a\right) + \delta \right)^2 \\
&= d^2\left(\frac{x+z}2,a\right) + 2d\left(\frac{x+z}2,a\right)\cdot \delta + \delta^2\\
&= d^2\left(\frac{x+z}2,a\right) + \delta\left(2d\left(\frac{x+z}2,a\right) +\delta \right)\\
&\leq d^2\left(\frac{x+z}2,a\right) + \delta\left(2r + r \right) \\
&\leq d^2\left(\frac{x+z}2,a\right) + \gamma,
\end{align*}
so
$$d^2(x,a) < d^2\left(\frac{x+z}2,a\right) + 2\gamma,$$
which we keep in mind, as this is what will get contradicted later. We also have that
$$\eps \leq d(x,y) \leq d(x,z) + d(z,y) < d(x,z) + \delta,$$
so $d(x,z) > \eps - \delta \geq \frac\eps2 \geq \frac\eps{2r} \cdot \alpha$. Since $d(x,a)=d(z,a)=\alpha$, using that $\eta$ is a monotone modulus of uniform convexity, we get that
$$d\left(\frac{x+z}2,a\right) \leq \left(1-\eta\left(\alpha,\frac\eps{2r}\right)\right) \alpha \leq \left(1-\eta\left(r,\frac\eps{2r}\right)\right)\alpha.$$
Since
$$\eps \leq d(x,y) \leq d(x,a) + d(y,a) = 2\alpha + (\beta-\alpha) < 2\alpha + \delta,$$
so
$$\alpha > \frac{\eps-\delta}2 \geq \frac\eps4,$$
we have that
$$d(x,a) - d\left(\frac{x+z}2,a\right) \geq \alpha - \left(1-\eta\left(r,\frac\eps{2r}\right)\right)\alpha = \alpha \eta\left(r,\frac\eps{2r}\right) \geq \frac\eps4 \eta\left(r,\frac\eps{2r}\right),$$
so 
$$d^2(x,a) \geq d^2\left(\frac{x+z}2,a\right) + \frac{\eps^2}{16}\eta^2\left(r,\frac{\eps}{2r}\right) = d^2\left(\frac{x+z}2,a\right) + 2\gamma,$$
a contradiction.
\end{proof}

\begin{remark}
An inspection of the proof shows that the only use of the axiom $(W4)$ may be replaced by that of the axiom $(W5)$, as introduced in \cite{KohNic}, which states that, for all $x$, $y$, $z \in X$ and $\lambda \in [0,1]$,
$$d(W(x,z,\lambda),W(y,z,\lambda))\leq d(x,y).$$
\end{remark}

With this result in mind, we may now show that the section of the square of the distance function is uniformly convex on bounded subsets.

\begin{corollary}
Let $\psi_\bullet$ be defined as in Theorem~\ref{mainthm}. Let $\eta :(0,\infty) \times (0,2] \to (0,\infty)$ be monotone and let $(X,d,W)$ be a $UCW$-hyperbolic space admitting $\eta$ as a modulus. Then, for any $\lambda \in [0,1]$, $r$, $\eps >0$ and any $a$, $x$, $y \in X$ with $d(x,a) \leq r$, $d(y,a) \leq r$, $d(x,y) \geq \eps $, we have that
\begin{align*}
d^2\left((1-\lambda)x+\lambda y,a\right) &\leq (1-\lambda) d^2(x,a) + \lambda d^2(y,a) - 2\min(\lambda, 1-\lambda)\psi_\eta(r,\eps)\\
&\leq (1-\lambda) d^2(x,a) + \lambda d^2(y,a) - 2\lambda(1-\lambda)\psi_\eta(r,\eps).
\end{align*}
\end{corollary}

\begin{proof}
Let $\lambda \in [0,1]$, $r$, $\eps >0$ and any $a$, $x$, $y \in X$ with $d(x,a) \leq r$, $d(y,a) \leq r$, $d(x,y) \geq \eps $. Assume w.l.o.g., using $(W3)$, that $\lambda \leq 1/2$, so $2\lambda \in [0,1]$.

Put $m:=\frac{x+y}2$. We first show that $W(x,y,\lambda)=W(x,m,2\lambda)$. Put $z:=W(x,m,2\lambda)$. We have that
$$d(z,x) = 2\lambda d(x,m) = \lambda d(x,y)$$
and
$$d(z,y) \leq d(z,m) + d(m,y) = \left(\frac12-\lambda\right)d(x,y) + \frac12d(x,y) = (1-\lambda)d(x,y),$$
but also that
$$d(z,y) \geq d(x,y) - d(z,x) = (1-\lambda)d(x,y).$$
Applying \cite[Proposition 5]{Leu07}, we obtain that $z=W(x,y,\lambda)$.

We thus get, using Proposition~\ref{quad} for the first inequality, that
\begin{align*}
d^2\left((1-\lambda)x+\lambda y,a\right) &\leq (1-2\lambda)d^2(x,a) + 2\lambda d^2(m,a) \\
&\leq (1-2\lambda)d^2(x,a) + 2\lambda\left(\frac12 d^2(x,a) + \frac12 d^2(y,a) - \psi_\eta(r,\eps)\right) \\
&= (1-\lambda) d^2(x,a) + \lambda d^2(y,a) - 2\lambda\psi_\eta(r,\eps)\\
&\leq (1-\lambda) d^2(x,a) + \lambda d^2(y,a) - 2\min(\lambda, 1-\lambda)\psi_\eta(r,\eps),
\end{align*}
and we are done.
\end{proof}

\section{Shadow sequences}\label{sec:shadow}

We will now produce some applications of our main result to the convergence of `shadow sequences', which, in this context (see, for example, \cite[p. 92]{BauCom17}), are projections of (quasi-)Fej\'er monotone sequences onto the set with respect to which they are (quasi-)Fej\'er monotone. Our study will be very much inspired by proof mining, and thus will be a quantitative one. One would expect that a quantitative version of a convergence theorem would exhibit a rate of convergence -- or, equivalently, of Cauchyness -- an indicator of how fast the sequence converges. It is known, however, that even simple situations like the elementary monotone convergence theorem for real sequences do not admit uniform rates of convergence. One is then driven to replace the Cauchyness property, say, of a sequence $(x_n)$ in a metric space $(X,d)$, by the next best thing, namely the classically equivalent, but constructively distinct property expressed by
$$\forall \eps> 0\ \forall g: \N \to \N\ \exists N \in \N\ \forall n,\ m \in [N, N+g(N)]\ d(x_n,x_m) \leq \eps,$$
which is notable for having been rediscovered in the 2000s by Terence Tao during his work in ergodic theory \cite{Tao08A}, and popularized through his blog \cite{Tao08}, where he treated the case of the monotone convergence theorem\footnote{I said `rediscovered', because even this monotone convergence case in itself had already been studied by Georg Kreisel in the 1950s, as it can be seen in the very paper \cite[pp. 49--50]{Kre} which introduced the program now known as proof mining.} and asked for suggestions for a proper name for the property. The name that ultimately stuck was {\bf metastability}, as suggested by Jennifer Chayes, and one may then speak of a `rate of metastability' as being an upper bound on the $N$ in the metastability statement, given in terms of the $\eps$ and the $g$ (and perhaps some other parameters of the problem at hand). Thus, we will state all our convergence results in terms of rates of metastability.

To express such rates, we introduce some notations. For all $f:\N \to \N$, we define $\wt{f} : \N \to \N$, for all $n$, by $\wt{f}(n):=n+f(n)$ and $f^M:\N\to\N$, for all $n$, by $f^M(n):=\max_{i \leq n} f(i)$; in addition, for all $n \in \N$, we denote by $f^{(n)}$ the $n$-fold composition of $f$ with itself.

\subsection{Sequences of reals}

We present below the now-standard quantitative, `metastable' version of the monotone convergence theorem, as pointed out by Tao \cite{Tao08} (see also \cite[pp. 30--32]{Koh08}).

\begin{proposition}\label{mmcp}
Let $b \geq 0$ and $(a_n)$ be a nonincreasing sequence in $[0,b]$. Let $\eps>0$ and $g:\N\to\N$. Then there is an $N \leq \wt{g}^{\left(\left\lceil\frac b\eps\right\rceil \right)}(0)$ such that, for all $n$, $m \in [N,N+g(N)]$, $|a_n-a_m|\leq\eps$.
\end{proposition}

\begin{proof}
Assume towards a contradiction that the conclusion is false. In particular, for all $i \leq \left\lceil\frac b\eps\right\rceil$, $a_{\wt{g}^{(i)}(0)} - a_{\wt{g}^{(i+1)}(0)} > \eps$. Then
$$a_0 \geq a_0 - a_{\wt{g}^{\left(\left\lceil\frac b\eps\right\rceil + 1\right)}(0)} = \sum_{i=0}^{\left\lceil\frac b\eps\right\rceil} \left(a_{\wt{g}^{(i)}(0)} - a_{\wt{g}^{(i+1)}(0)} \right)> \left\lceil\frac b\eps\right\rceil\cdot\eps \geq b,$$
a contradiction.
\end{proof}

The following is a slight generalization of the monotone convergence theorem, which has proven to be quite useful in nonlinear analysis.

\begin{proposition}[{Tan and Xu, \cite[Lemma 1]{TanXu}}]
Let $(a_n)$ and $(\delta_n)$ be sequences of non-negative reals such that $\sum \delta_n < \infty$ and, for all $n$, $a_{n+1} \leq a_n + \delta_n$. Then $(a_n)$ is convergent.
\end{proposition}

We now present, for the first time, the full quantitative version of the above theorem (particular cases where the limit inferior of $(a_n)$ is known to be zero have been previously treated in \cite{Leu10, Leu14,Sip17}, see also the recent abstract version of those in \cite{FirLeuXX}).

\begin{proposition}\label{mmcp2}
Let $b \geq 0$ and $(a_n)$ be a nonincreasing sequence in $[0,b]$. Let $(\delta_n)$ be a sequence of non-negative reals such that, for all $n$, $a_{n+1} \leq a_n + \delta_n$. Let $\gamma : (0, \infty) \to \N$ be such that, for all $\eps >0$ and all $m \geq \gamma(\eps)$,
$$\sum_{i \in [\gamma(\eps),m)} \delta_i \leq \eps.$$

Let $\eps>0$ and $g:\N\to\N$. Then there is an
$$N \in \left[ \gamma\left(\frac\eps4\right), \wt{g^M}^{\left(\left\lceil\frac {2b}\eps\right\rceil \right)}\left(\gamma\left(\frac\eps4\right)\right) \right]$$
such that, for all $n$, $m \in [N,N+g(N)]$, $|a_n-a_m|\leq\eps$.
\end{proposition}

\begin{proof}
Firstly, we have that, for all $n \geq \gamma\left(\frac\eps4\right)$ and all $m \geq n$,
$$a_m \leq a_n + \sum_{i \in [n,m)} \delta_i \leq a_n + \sum_{i \in \left[\gamma\left(\frac\eps4\right),m\right)} \delta_i \leq a_n + \frac\eps4.$$
Define, now, the function $g^*: \N \to \N$ by putting, for all $n$,
$$g^*(n):= \argmax_{q \in [0,g(n)]} |a_n - a_{n+q}|.$$
We have that, for all $n$, $g^*(n) \leq g(n) \leq g^M(n)$.

Put, for all $i \leq \left\lceil\frac {2b}\eps\right\rceil$, $N_i:=\wt{g^*}^{(i)}\left(\gamma\left(\frac\eps4\right)\right)$. Assume towards a contradiction that, for all $i \leq \left\lceil\frac {2b}\eps\right\rceil$, $|a_{N_i} - a_{N_{i+1}}| > \frac\eps2$. But, since, as we have shown earlier, for each $i$, $a_{N_{i+1}} \leq a_{N_i}+\frac\eps4$, we must have that, for each $i$, $a_{N_i} - a_{N_{i+1}} > \frac\eps2$, so
$$a_{N_0} \geq a_{N_0} - a_{N_{\left\lceil\frac {2b}\eps\right\rceil+1}} = \sum_{i=0}^{\left\lceil\frac {2b}\eps\right\rceil} (a_{N_i} - a_{N_{i+1}}) > \left(\left\lceil\frac {2b}\eps\right\rceil + 1 \right) \cdot \frac\eps2 > \left\lceil\frac {2b}\eps\right\rceil \cdot \frac\eps2 \geq b,$$
a contradiction. Thus, there is a $j \leq \left\lceil\frac {2b}\eps\right\rceil$ such that $|a_{N_j} - a_{N_{j+1}}| \leq \frac\eps2$.

We will take $N:=N_j$ and show that it satisfies the desired properties. Clearly, $N \geq \gamma\left(\frac\eps4\right)$ and, for all $n$, $m \in [N, N+g(N)]$,
$$|a_n - a_m| \leq |a_N - a_n| + |a_N - a_m| \leq 2|a_N - a_{N+g^*(N)}| = 2|a_{N_j} - a_{N_{j+1}}| \leq 2 \cdot \frac\eps2 = \eps.$$ 

It remains to be shown that
$$N \leq \wt{g^M}^{\left(\left\lceil\frac {2b}\eps\right\rceil \right)}\left(\gamma\left(\frac\eps4\right)\right).$$

We show that, for all $i$, $n \in \N$,
$$\wt{g^*}^{(i)}(n)\leq  \wt{g^M}^{(i)}(n),$$
from which it will follow that
$$N = N_j = \wt{g^*}^{(j)}\left(\gamma\left(\frac\eps4\right)\right) \leq \wt{g^M}^{(j)}\left(\gamma\left(\frac\eps4\right)\right) \leq \wt{g^M}^{\left(\left\lceil\frac {2b}\eps\right\rceil \right)}\left(\gamma\left(\frac\eps4\right)\right).$$

Let $n \in \N$. We prove the claim by induction on $i$. The base case is obvious. Assume, now, that the claim holds for an $i$ and we prove it for $i+1$. We have that
$$\wt{g^*}^{(i+1)}(n)=\wt{g^*}\left(\wt{g^*}^{(i)}(n)\right) \leq \wt{g^M}\left(\wt{g^*}^{(i)}(n)\right) \leq \wt{g^M}\left(\wt{g^M}^{(i)}(n)\right) = \wt{g^M}^{(i+1)}(n),$$
and we are done.
\end{proof}

\subsection{Fej\'er monotone sequences}

The following is the generalization from Hilbert spaces to complete $UCW$-hyperbolic spaces of \cite[Proposition 5.7]{BauCom17} (see also \cite[Lemma 3.2]{TakToy03}).

\begin{theorem}
Let $\psi_\bullet$ be defined as in Theorem~\ref{mainthm}. Let $\eta :(0,\infty) \times (0,2] \to (0,\infty)$ be monotone and let $(X,d,W)$ be a complete $UCW$-hyperbolic space admitting $\eta$ as a modulus. Let $S$ be a closed, convex, nonempty subset of $X$, so that the projection $P_S:X \to S$ is well-defined, and let $(x_n)$ be a sequence in $X$ which is Fej\'er monotone with respect to $S$. Let $b\geq 0$ be such that $d(x_0,P_Sx_0) \leq b$. Then, for all $\eps>0$ and $g:\N\to\N$, there is an $N \leq \wt{g}^{\left(\left\lceil\frac{b^2}{\psi_\eta(b,\eps)}\right\rceil \right)}(0)$ such that, for all $n$, $m \in [N,N+g(N)]$, $d(P_Sx_n,P_Sx_m)\leq\eps$.
\end{theorem}

\begin{proof}
Firstly, we have that, for all $n \in \N$, using the projection property and then the Fej\'er monotonicity, that
$$d(P_Sx_n,x_n) \leq d(P_Sx_0,x_n) \leq d(P_Sx_0,x_0) \leq b.$$
Then, for all $n$, $m \in \N$ with $n < m$, we have, using the Fej\'er monotonicity, that
$$d(P_Sx_n,x_m) \leq d(P_Sx_n,x_n) \leq b.$$
We will show that, for all $\eps > 0$ and all $n$, $m \in \N$ such that $n < m$ and $d(P_Sx_n,P_Sx_m) \geq \eps$, we have that
$$\psi_\eta(b,\eps) \leq \frac12 (d^2(P_Sx_n,x_n) - d^2(P_Sx_m,x_m)).$$
Let $\eps$, $n$, $m$ be as stated. Using the projection property, the convexity of $S$ and Theorem~\ref{mainthm} (since $d(P_Sx_m,x_m) \leq b$ and $d(P_Sx_n,x_m) \leq b$), we have that
$$d^2(P_Sx_m,x_m) \leq d^2 \left(\frac{P_Sx_m+P_Sx_n}2, x_m\right) \leq \frac12 d^2(P_Sx_m,x_m) + \frac12 d^2(P_Sx_n,x_m) - \psi_\eta(b,\eps).$$
Using the (already established) fact that $d(P_Sx_n,x_m) \leq d(P_Sx_n,x_n)$, we get our conclusion.

We now show that the sequence $(d^2(P_Sx_n,x_n))$ is a nonincreasing sequence in $[0,b^2]$. Let $n \in \N$. If $P_Sx_n = P_Sx_{n+1}$, then, using the Fej\'er monotonicity, $d(P_Sx_{n+1},x_{n+1}) = d(P_Sx_n,x_{n+1}) \leq d(P_Sx_n,x_n)$. If $P_Sx_n \neq P_Sx_{n+1}$, then, by the previous property, we have that
$$0 < \psi_\eta(b,d(P_Sx_n,P_Sx_{n+1})) \leq \frac12 (d^2(P_Sx_n,x_n) - d^2(P_Sx_{n+1},x_{n+1}))$$
and we are done.

Let now $\eps>0$ and $g:\N \to \N$. By Proposition~\ref{mmcp}, there is an $N \leq \wt{g}^{\left(\left\lceil\frac {b^2}{\psi_\eta(b,\eps)}\right\rceil \right)}(0)$ such that for all $n$, $m \in [N,N+g(N)]$, $|d^2(P_Sx_n,x_n) - d^2(P_Sx_m,x_m)|\leq\psi_\eta(b,\eps)$.

This $N$ will be our desired $N$. Let now $n$, $m \in [N,N+g(N)]$. We want to show that $d(P_Sx_n,P_Sx_m)\leq\eps$. Assume w.l.o.g. that $n < m$ and assume towards a contradiction that $d(P_Sx_n,P_Sx_m) > \eps$. By the property shown before, we have that
$$0 < 2\psi_\eta(b,\eps) \leq d^2(P_Sx_n,x_n) - d^2(P_Sx_m,x_m) \leq  \psi_\eta(b,\eps),$$
a contradiction.
\end{proof}

\begin{corollary}
Let $(X,d,W)$ be a complete $UCW$-hyperbolic space. Let $S$ be a closed, convex, nonempty subset of $X$, so that the projection $P_S:X \to S$ is well-defined, and let $(x_n)$ be a sequence in $X$ which is Fej\'er monotone with respect to $S$. Then the sequence $(P_Sx_n)$ is convergent.
\end{corollary}

\begin{remark}
If $T$ is a nonexpansive self-mapping of a complete $UCW$-hyperbolic space $(X,d,W)$, then we might take $S:=Fix(T)$ if that set is nonempty (e.g. if $X$ is bounded and nonempty). A sequence $(x_n)$ is Fej\'er monotone with respect to this set if, in particular, it is assumed to be a {\bf Mann iteration} of $T$, that is, there is a sequence $(\alpha_n) \se [0,1]$ such that, for each $n$,
$$x_{n+1} = \alpha_n Tx_n + (1-\alpha_n) x_n.$$
\end{remark}

\subsection{Quasi-Fej\'er monotone sequences}

The following is the generalization from complete CAT(0) spaces to complete $UCW$-hyperbolic spaces of \cite[Lemma 5]{SaeYot20}.

\begin{theorem}
Let $\psi_\bullet$ be defined as in Theorem~\ref{mainthm}. Let $\eta :(0,\infty) \times (0,2] \to (0,\infty)$ be monotone and let $(X,d,W)$ be a complete $UCW$-hyperbolic space admitting $\eta$ as a modulus. Let $S$ be a closed, convex, nonempty subset of $X$, so that the projection $P_S:X \to S$ is well-defined, and let $(x_n)$ be a sequence in $X$ and $(\delta_n)$ be a sequence of non-negative reals such that, for all $q \in S$ and for all $n$, $d(x_{n+1},q) \leq d(x_n,q) + \delta_n$ (a property to which we will refer in the proof as {\it quasi-Fej\'er monotonicity}). Let $b$, $B >0$ be such that $d(x_0,P_Sx_0) \leq b$ and $\sum \delta_n \leq B$. Set $C:=2b+3B$. Let $\gamma : (0, \infty) \to \N$ be such that, for all $\eps >0$ and all $m \geq \gamma(\eps)$,
$$\sum_{i \in [\gamma(\eps),m)} \delta_i \leq \eps.$$
Then, for all $\eps>0$ and $g:\N\to\N$, there is an
$$N \leq \wt{g^M}^{\left(\left\lceil\frac {2C}{\psi_\eta(C,\eps)}\right\rceil \right)}\left(\gamma\left(\frac{\psi_\eta(C,\eps)}{4C}\right)\right)$$
such that, for all $n$, $m \in [N,N+g(N)]$, $d(P_Sx_n,P_Sx_m)\leq\eps$.
\end{theorem}

\begin{proof}
Firstly, we have that, for all $n \in \N$, using the projection property and then the quasi-Fej\'er monotonicity, that
$$d(P_Sx_n,x_n) \leq d(P_Sx_0,x_n) \leq d(P_Sx_0,x_0) +B \leq b + B \leq C.$$
Then, for all $n$, $m \in \N$ with $n < m$, we have, using the quasi-Fej\'er monotonicity, that
$$d(P_Sx_n,x_m)\leq d(P_Sx_n,x_n) +\sum_{i \in [n,m)} \delta_i \leq d(P_Sx_n,x_n) + B \leq b+2B \leq C$$
and, in addition, that
\begin{align*}
d^2(P_Sx_n,x_m) &\leq \left(d(P_Sx_n,x_n) +\sum_{i \in [n,m)} \delta_i\right)^2 \\
&= d^2(P_Sx_n,x_n) +2 \cdot d(P_Sx_n,x_n) \cdot\left(\sum_{i \in [n,m)} \delta_i \right)+ \left(\sum_{i \in [n,m)} \delta_i \right)^2 \\
&\leq d^2(P_Sx_n,x_n) + 2(b+B)\left(\sum_{i \in [n,m)} \delta_i \right) + B\left(\sum_{i \in [n,m)} \delta_i \right)\\
&= d^2(P_Sx_n,x_n) + C\left(\sum_{i \in [n,m)} \delta_i \right).
\end{align*}
We will show that, for all $\eps > 0$ and all $n$, $m \in \N$ such that $n < m$ and $d(P_Sx_n,P_Sx_m) \geq \eps$, we have that
$$\psi_\eta(C,\eps) \leq \frac12 \left(d^2(P_Sx_n,x_n) - d^2(P_Sx_m,x_m) + C\left(\sum_{i \in [n,m)} \delta_i \right)\right).$$
Let $\eps$, $n$, $m$ be as stated. Using the projection property, the convexity of $S$ and Theorem~\ref{mainthm} (since $d(P_Sx_m,x_m) \leq C$ and $d(P_Sx_n,x_m) \leq C$), we have that
$$d^2(P_Sx_m,x_m) \leq d^2 \left(\frac{P_Sx_m+P_Sx_n}2, x_m\right) \leq \frac12 d^2(P_Sx_m,x_m) + \frac12 d^2(P_Sx_n,x_m) - \psi_\eta(C,\eps).$$
Using the (already established) fact that
$$d^2(P_Sx_n,x_m) \leq d^2(P_Sx_n,x_n) + C\left(\sum_{i \in [n,m)} \delta_i \right),$$
we get our conclusion.

We now show that, for all $n$,
$$d^2(P_Sx_{n+1},x_{n+1}) \leq d^2(P_Sx_n,x_n) + C\delta_n.$$
Let $n \in \N$. If $P_Sx_n = P_Sx_{n+1}$, then $d^2(P_Sx_{n+1},x_{n+1}) = d^2(P_Sx_n,x_{n+1}) \leq d^2(P_Sx_n,x_n) + C\delta_n$. If $P_Sx_n \neq P_Sx_{n+1}$, then, by the previous property, we have that
$$0 < \psi_\eta(C,d(P_Sx_n,P_Sx_{n+1})) \leq \frac12 (d^2(P_Sx_n,x_n) - d^2(P_Sx_{n+1},x_{n+1}) + C\delta_n)$$
and we are done.

Let now $\eps>0$ and $g:\N \to \N$. By Proposition~\ref{mmcp2}, there is an
$$N \in \left[ \gamma\left(\frac{\psi_\eta(C,\eps)}{4C}\right), \wt{g^M}^{\left(\left\lceil\frac {2C}{\psi_\eta(C,\eps)}\right\rceil \right)}\left(\gamma\left(\frac{\psi_\eta(C,\eps)}{4C}\right)\right) \right]$$
such that for all $n$, $m \in [N,N+g(N)]$, $|d^2(P_Sx_n,x_n) - d^2(P_Sx_m,x_m)|\leq\psi_\eta(C,\eps)$.

This $N$ will be our desired $N$. Let now $n$, $m \in [N,N+g(N)]$. We want to show that $d(P_Sx_n,P_Sx_m)\leq\eps$. Assume w.l.o.g. that $n < m$ and assume towards a contradiction that $d(P_Sx_n,P_Sx_m) > \eps$. By the property shown before, we have that
$$2\psi_\eta(b,\eps) \leq d^2(P_Sx_n,x_n) - d^2(P_Sx_m,x_m)+ C\left(\sum_{i \in [n,m)} \delta_i \right) \leq  \psi_\eta(b,\eps)+ C\left(\sum_{i \in [n,m)} \delta_i \right).$$
Since 
$$n \geq N \geq \gamma\left(\frac{\psi_\eta(C,\eps)}{4C}\right),$$
we have that
$$\sum_{i \in [n,m)} \delta_i \leq \sum_{i \in \left[\gamma\left(\frac{\psi_\eta(C,\eps)}{4C}\right),m\right)} \delta_i \leq \frac{\psi_\eta(C,\eps)}{4C},$$
so
$$0 < 2\psi_\eta(C,\eps) \leq \psi_\eta(C,\eps)+ C \cdot \frac{\psi_\eta(C,\eps)}{4C} = \frac54\psi_\eta(C,\eps),$$
a contradiction.
\end{proof}

\begin{corollary}
Let $(X,d,W)$ be a complete $UCW$-hyperbolic space. Let $S$ be a closed, convex, nonempty subset of $X$, so that the projection $P_S:X \to S$ is well-defined, and let $(x_n)$ be a sequence in $X$ and $(\delta_n)$ be a sequence of non-negative reals such that $\sum \delta_n < \infty$ and, for all $q \in S$ and for all $n$, $d(x_{n+1},q) \leq d(x_n,q) + \delta_n$. Then the sequence $(P_Sx_n)$ is convergent.
\end{corollary}

\begin{remark}
If $(\delta_n)$ is a sequence of non-negative reals such that $\delta_n \to 0$ and $T$ is a self-mapping of a complete $UCW$-hyperbolic space $(X,d,W)$ which is asymptotically nonexpansive w.r.t. $(\delta_n)$, then we might take $S:=Fix(T)$ if that set is nonempty (e.g. if $X$ is bounded and nonempty). A sequence $(x_n)$ is quasi-Fej\'er monotone with respect to this set if, in particular, $\sum \delta_n < \infty$ and the sequence is assumed to be {\bf bounded} and to be a {\bf Schu iteration} of $T$, that is, there is a sequence $(\alpha_n) \se [0,1]$ such that, for each $n$,
$$x_{n+1} = \alpha_n T^n x_n + (1-\alpha_n) x_n.$$
\end{remark}

\section{Proximal mappings}\label{sec:proximal}

For this section, we will fix a complete $UCW$-hyperbolic space $(X,d,W)$ and a monotone modulus $\eta$ of it. The following notions are standard in convex optimization.

\begin{definition}
Let $f : X \to \R \cup \{\infty\}$. We say that $f$ is:
\begin{itemize}
\item {\bf proper} if there is an $x \in X$ with $f(x) \neq \infty$;
\item {\bf convex} if, for all $x$, $y \in X$ and all $\lambda \in [0,1]$, $f((1-\lambda)x+\lambda y) \leq (1-\lambda)f(x) + \lambda f(y)$;
\item {\bf lower semicontinuous (lsc)} if, for all $x \in X$, $r \in \R$ with $r<f(x)$ there is a $\delta >0$ such that, for all $z \in X$ with $d(x,z) \leq \delta$, $r<f(z)$.
\end{itemize}
For any $\lambda > 0$, we define the function $\lambda f : X \to \R \cup \{\infty\}$ by putting, for any $x \in X$, $(\lambda f)(x):=\lambda \cdot f(x)$, and it is clear that $\lambda f$ is proper, convex and/or lsc iff $f$ is so.
\end{definition}

\begin{proposition}[{\cite[Proposition 2.2]{KohLeu10}}]\label{kohleus}
Let $(C_n)_{n \in \N}$ be a non-increasing sequence of bounded, closed, convex, nonempty subsets of $X$. Then $\bigcap_{n \in \N} C_n \neq \emptyset$.
\end{proposition}

\begin{proposition}\label{bac1}
Let $f : X \to \R \cup \{\infty\}$ be a convex, lsc function. Then $f$ is bounded from below on bounded sets.
\end{proposition}

\begin{proof}
The proof follows the same lines as that of \cite[Lemma 2.2.12]{Bac14}, using Proposition~\ref{kohleus} instead of \cite[Proposition 2.1.14]{Bac14}.
\end{proof}

\begin{proposition}\label{bac2}
Let $f : X \to \R \cup \{\infty\}$ be a convex, lsc function and $a \in X$. Then there is a $k \in \N$ such that, for any $x \in X$,
$$f(x) \geq -k(d(x,a) +1).$$
\end{proposition}

\begin{proof}
The proof follows the same lines as that of \cite[Lemma 2.2.13]{Bac14}, using Proposition~\ref{bac1} instead of \cite[Lemma 2.2.12]{Bac14}.
\end{proof}

\begin{proposition}
Let $f : X \to \R \cup \{\infty\}$ be a proper, convex, lsc function and $a \in X$. Define $h: X \to \R \cup \{\infty\}$ by setting, for any $x \in X$,
$$h(x):= f(x) + \frac12 d^2(x,a).$$
Then $h$ has a unique minimizer.
\end{proposition}

\begin{proof}
We first show that $h$ is bounded from below. By Proposition~\ref{bac2}, there is a $k \in \N$ such that, for any $x \in X$,
$$h(x) \geq -k(d(x,a) +1) + \frac12 d^2(x,a),$$
and the right hand side is clearly bounded from below by the basic properties of quadratic functions.

Let $m \in \R$ be the infimum of $h$ (it is not $\infty$, since $f$ is proper). We now show the existence of an $x \in X$ with $h(x)=m$. Again by the properties of quadratic functions, there is a $r > 0$ such that, for any $u \in \R$ with $\frac12 u^2 - k(u+1) \leq m+1$, we have that $u \leq r$. Therefore, for any $x \in X$ with $h(x) \leq m+1$, we have that $d(x,a) \leq r$. Put, for any $n \in \N$,
$$C_n := \left\{ x \in X \mid h(x) \leq m + \frac1{n+1} \right\}.$$
We have that $(C_n)_{n \in \N}$ is a non-increasing sequence of bounded (this one from the reasoning above), closed, convex, nonempty subsets of $X$, so, by Proposition~\ref{kohleus}, there is a $x \in \bigcap_{n \in \N} C_n$ and we must have that $h(x)=m$.

Therefore, $m$ is the minimum of $h$. We now show the uniqueness of the minimizer. Let $\psi_\bullet$ be defined as in Theorem~\ref{mainthm}. Assume there are $x \neq y$ with $h(x)=h(y)=m$. Put $\eps:=d(x,y)>0$ and $r:=\max(d(x,a),d(y,a)) >0$. Then
\begin{align*}
h\left(\frac{x+y}2\right) &= f\left(\frac{x+y}2\right) + \frac12 d^2\left(\frac{x+y}2,a\right) \\
&\leq \frac12(f(x)+f(y)) + \frac14 d^2(x,a) + \frac14 d^2(y,a) - \frac12\psi_\eta(r,\eps) = m - \frac12\psi_\eta(r,\eps) < m,
\end{align*}
a contradiction.
\end{proof}

This allows us to extend proximal minimization, an established tool of convex optimization, to this uniformly convex hyperbolic setting (the first such nonlinear generalization, namely to complete CAT(0) spaces, was introduced by Jost \cite{Jos95}). Concretely, we may now define, for any proper, convex, lsc function $f : X \to \R \cup \{\infty\}$, the {\it proximal mapping of $f$}, $\Prox_f : X \to X$, by putting, for any $a \in X$,
$$\Prox_fa := \argmin_{x \in X} \left(f(x) + \frac12 d^2(x,a)\right),$$
and, for any $\lambda > 0$, the {\it proximal mapping of $f$ of order $\lambda$}, $\Prox^\lambda_f : X \to X$, by putting, for any $a \in X$,
\begin{align*}
\Prox^\lambda_fa &:= \Prox_{\lambda f}a \\
&= \argmin_{x \in X} \left(\lambda f(x) + \frac12 d^2(x,a)\right)\\
&= \argmin_{x \in X} \left(f(x) + \frac1{2\lambda} d^2(x,a)\right).
\end{align*}

\begin{proposition}
Let $f : X \to \R \cup \{\infty\}$ be a proper, convex, lsc function and $a \in X$. Then $a$ is a minimizer of $f$ iff it is a fixed point of $\Prox_f$.
\end{proposition}

\begin{proof}
From left to right, we have that, for any $x \in X$,
$$f(a) + \frac12 d^2(a,a) = f(a) \leq f(x) \leq f(x) + \frac12 d^2(x,a),$$
so
$$a = \argmin_{x \in X} \left(f(x) + \frac12 d^2(x,a)\right) = \Prox_fa.$$
From right to left, we have that, for any $w \in X$,
$$f(a) = f(a) + \frac12 d^2(a,a) = f(\Prox_fa) + \frac12 d^2(\Prox_fa,a) \leq f(w) + \frac12 d^2(w,a),$$
so, for any $x \in X$ and any $t \in (0,1)$, using the convexity of $f$,
$$f(a) \leq f((1-t)a+tx) + \frac12 d^2((1-t)a+tx,a) \leq (1-t)f(a) + tf(x) + \frac12 t^2 d^2(x,a),$$
so, by dividing by $t$,
$$f(a) \leq f(x) + \frac12 t d^2(x,a),$$
and, by taking $t \to 0$, we obtain the conclusion.
\end{proof}

In \cite{BacKoh18}, the authors of that paper explore a plethora of results around proximal mappings with Young functions in uniformly convex Banach spaces. Investigating their extension to complete $UCW$-hyperbolic spaces would be a natural continuation of the present paper.

\end{document}